\documentclass{article}
\usepackage{amsmath}

\newtheorem{thm}{Theorem}[section]
\newtheorem{cor}[thm]{Corollary}
\newtheorem{prop}[thm]{Proposition}

\newtheorem{lem}[thm]{Lemma}
\newtheorem{Def}[thm]{Definition}
\newtheorem{rem}[thm]{Remark}

\newtheorem{ex}[thm]{Example}

\newcommand{\be}{\begin{equation}}
    \newcommand{\ee}{\end{equation}}
\newcommand{\ben}{\begin{enumerate}}
    \newcommand{\een}{\end{enumerate}}
\newcommand{\beq}{\begin{eqnarray}}
    \newcommand{\eeq}{\end{eqnarray}}
\newcommand{\beqn}{\begin{eqnarray*}}
    \newcommand{\eeqn}{\end{eqnarray*}}

\newcommand{\pa}{\partial}

\newcommand{\qed}{\hspace*{\fill}Q.E.D.}  

\begin{document}
    \title{Global Metrizability on Sprays of Scalar Curvature}
    \author{Xingzhi Duanmu  and  Guojun Yang\footnote{Corresponding Author } }
    \date{}
    \maketitle

    \begin{abstract}
  Sprays of scalar curvature constitute an important class of
  sprays, and such a class includes all two-dimensional sprays. In
  this paper, we consider the global  metrizability of certain sprays  on
  a manifold under some curvature  and topological
  conditions. We prove that, for a regular spray {\bf G} of scalar curvature
   with almost everywhere nonzero Ricci curvature on a manifold $M$ of dimension $n\ge 3$, {\bf G} is
  globally metrizable if and only if it is locally metrizable,
  provided that $M$ has trivial first de Rham cohomology group.
  Further, we characterize the global metizability of a class of
  two-dimensional singular Berwald sprays on a manifold with trivial first de Rham
  cohomology group. Meanwhile, on a cylinder $S^1\times {\bf R}$ (with non-trivial first de Rham
  cohomology group), we construct a family of
  singular Berwald sprays which are locally metrizable but not globally
   metrizable. Finally, we construct some examples of  two-dimensional
   sprays on  cylinders or spheres with special
   properties.


        {\bf MR(2000) Subject Classification: }
         53C60, 53B40

    \end{abstract}

    \section{Introduction}

   Spray geometry  studies the properties of  path spaces (consisting of geodesics),
and it is more general than Finsler geometry. A spray ${\bf G}$ on
a manifold $M$ is a family of compatible second order ODEs which
define a special vector field on a conical region $\mathcal{C}$ of
 $T^oM:=TM\setminus \{0\}$. When $\mathcal{C}=T^oM$, {\bf G} is called regular;
 when $\mathcal{C}_x\ne T^o_xM$ for any $x\in M$, {\bf G} is called
 singular.
Every Finsler metric induces a natural spray but there are many
sprays which cannot be induced by any Finsler metric (\cite{BM2,
EM, Shen6, Yang1, Yang2, Yang3}). So a natural issue is to find a
bridge connecting sprays and Finsler metrics, which is called the
metrizability problem.

The metrizability problem for a spray ${\bf G}$ seeks for a
Finsler metric whose spray is just {\bf G}. So a natural target is
to determine certain curvature conditions under which a given
spray is Finsler-metrizable or not.  Quite a few papers
concentrate on this study. One hand, some research jobs focus on
the conditions such that a given spray is not metrizable.
 In \cite{Mu}, Z. Muzsnay gives
some sprays which are not Finsler-metrizable under some conditions
satisfied by the holonomy distribution.  In \cite{EM}, S.G.
Elgendi and Z. Muzsnay discuss a more general class of sprays and
prove that they are not metrizabile by using the holonomy
distribution.   What's more important, on the other hand, is to
characterize certain given class of sprays to be metrizable.  In
\cite{BM2}, I. Bucataru and Z. Muzsnay give necessary and
sufficient conditions for sprays of scalar curvature with nonzero
Ricci curvature to be metrizable. In \cite{Yang2}, G. Yang uses
the metrizability condition for a spray of isotropic curvature to
give the local structure of projectively flat Berwald spray of
isotropic curvature, and in \cite{Yang3}, studies more
metrizability problems of some classes of sprays.

The jobs mentioned above on the metrizability of certain  sprays
essentially focus on the local construction, which is independent
of the topological structure of the concerned manifold on which a
spray is defined. A natural problem is, for a spray manifold
$(M,{\bf G})$, under what topological conditions on $M$ and
curvature conditions on {\bf G}, the spray {\bf G} is globally
metrizable if it is locally metrizable? In \cite{CMS}, M. Crampin,
T. Mestdag and D. J. Saunders show a global result on how to
construct a global Finsler metric projectively related to a spray
{\bf G} under certain curvature condition of {\bf G} and the
assumption of vanishing second C${\rm\breve{e}}$ch cohomology
group of the manifold on which {\bf G} is defined. Thus, for a
spray manifold $(M,{\bf G})$, where {\bf G} is locally metrizable
on $M$, the topological structure of $M$ should be closely related
to the global metrizability of {\bf G} on $M$.

In this paper, we focus on the  local and global metrizability of
sprays of scalar curvature and some special two-dimensional
Berwald sprays. For a spray manifold $(M,{\bf G})$, we always
assume that $M$ is connected and is of second countable axiom.

Sprays of scalar (resp. isotropic, constant)  curvature, which are
introduced in  \cite{Shen6} (resp. \cite{LS}, \cite{Yang2}), are
an important class of sprays. It is the natural generalization of
the class of Finsler metrics of scalar (resp. isotropic, constant)
flag curvature. A spray ${\bf G}$ is said to be of {\it scalar
curvature} if  its Riemann curvature $R^i_{\ k}$ satisfies
(\cite{Shen6}).
 \be\label{yg1}
  R^i_{\ k}=R\delta^i_k-\tau_ky^i,
  \ee
where $R=R(x,y)$ and $\tau_k=\tau_k(x,y)$ are some positively
homogeneous functions.
 Denote by $Ric$ the Ricci curvature of a spray. In (\ref{yg1}), we have $Ric=(n-1)R$ and
 $R=\tau_0 (=\tau_ry^r)$, where $n$ is the dimension of {\bf G}.
 We first have
the following theorem.

 \begin{thm}\label{Thm01}
  Let $(M, {\bf G})$ be a regular spray manifold with the dimension $n\ge 3$, where {\bf G}
   is of scalar curvature
   with  $Ric\ne 0$ almost everywhere on $T^oM$, and
   $M$ has trivial first de Rham cohomology group. Then {\bf G} is
  globally metrizable if and only if it is locally metrizable.
    \end{thm}

In Theorem \ref{Thm01}, since $T^oM$ is connected, the globally
defined Finsler metric $L$ on $M$
  satisfies $L>0$  or  $L<0$ on $T^oM$
  everywhere. Since $Ric$ is positively homogeneous, $T^oM$ can be replaced by the
  (projective) sphere bundle $SM$. If $M$ is simply connected in Theorem \ref{Thm01}, then the first de Rham cohomology
 group is trivial.

We do not know whether Theorem \ref{Thm01} is true for the case
$n=2$, or the case that {\bf G} is a singular spray defined on a
connected conical region $\mathcal{C}$ of $T^oM$ (with $n\ge 2$)
(see the connectivity of $\mathcal{C}$ in Section \ref{pre} and
Lemma \ref{lem21} below). So a natural problem is whether there
are such sprays which are locally metrizable but not globally
metrizable. However, on certain two-dimensional manifold with
non-trivial first de Rham cohomology group, there exist singular
sprays which are locally metrizable but not globally metrizable.
We have the following example (for general case, see Example
\ref{ex001}).

 \begin{ex}
 Let \( M \) be the cylinder \( S^1 \times \mathbf{R} \) (or torus $S^1\times S^1$),
  which has non-trivial first de Rham cohomology group.
 At an arbitrary point $(x^1,x^2)\in M$, let
   $$\phi(x^1, x^2) = 5 + 4 \sin x^1 \sin x^2,\ \ \ \
    \psi(x^1, x^2) = -2 \cos x^1 \cos x^2.
    $$
    Then the spray {\bf G} give by
  $$
    G^1=\phi(x^1,x^2)(y^1)^2,\ \ \ \ \ G^2=\psi(x^1,x^2)(y^2)^2
  $$
  is globally defined on $M$. It can be verified that {\bf G}
   is locally metrizable but not globally metrizable on the conical region $\mathcal{C}$,
   where, on each tangent space $T_xM$ with $x=(x^1,x^2)$,
   $$
 \mathcal{C}_x:=\{(y^1,y^2)\ \big|\ (y^1,y^2)\in T_xM,\ y^1y^2\ne0\}.
   $$
 \end{ex}

 In \cite{BM2}, the authors discuss the local metrizability of
 sprays of scalar curvature (Lemma \ref{Plem1}). Based on this,
 we have the following result.

 \begin{thm}\label{Thm02}
  Let $(M, {\bf G})$ be an $n$-dimensional spray manifold, where {\bf G}
  is defined on a connected conical region $\mathcal{C}$ and it is of scalar curvature
  with  $Ric\ne 0$  everywhere on $\mathcal{C}$, and $M$ has trivial first de Rham cohomology group.
 Further, suppose that {\bf G} is regular with $n\ge 3$, or {\bf G} is singular with $n\ge 2$. Then
  {\bf G}
 is globally metrizable on $\mathcal{C}$ if and only if
  \be\label{Met1}
 \left( \frac{\tau_i}{R} \right)_{.j} = \left( \frac{\tau_j}{R} \right)_{.i}, \quad
 \left( \frac{\tau_i}{R} \right)_{;j} = 0, \quad \det \left( \left( \frac{\tau_i}{R} \right)_{.j}
 + \frac{2\tau_i \tau_j}{R^2} \right) \neq 0,
 \ee
 where $R$ and $\tau_k$ are determined by (\ref{yg1}).
 \end{thm}

 In Theorem \ref{Thm02}, we
 cannot assume that  $Ric\ne 0$ almost everywhere on $\mathcal{C}$
with (\ref{Met1}) holding where $Ric\ne 0$ (see the sprays in
Theorem \ref{Thm03}), and note that here a regular spray is not  a
special case of a singular spray.

If considering a singular spray  defined on a conical region
$\mathcal{C}$ with multiple connected components,  we have some
results on the global metrizability of a special class of
two-dimensional Berwald sprays (Theorem \ref{Thm03} and Corollary
\ref{cor0l} below), where a Berwald spray means that its spray
coefficients are polynomials (every Berwald metric induces a
Berwald spray). In Section \ref{Example} below, we give some
Finsler metrics or sprays on the cylinder $S^1\times {\bf R}$ or
the sphere $S^2$, and discuss some of their properties.

    \section{Preliminaries}\label{pre}

   Let $M$ be a connected $n$-dimensional manifold.
 We say that
$\mathcal{C}_x$ is a conical region of $T_xM\setminus \{0\}$ if
$y\in \mathcal{C}_x$,  we have $\lambda y\in \mathcal{C}_x$ for
any $\lambda>0$. Then $\mathcal{C}=\mathcal{C}(M):=\cup_{x\in M}
\mathcal{C}_x$ is called a conical region of $TM\setminus \{0\}$,
if $\mathcal{C}$ is an open set in $TM\setminus \{0\}$. For  a
subset $U(\subset M)$, denote by $\mathcal{C}(U)$ the
corresponding conical region on $U$.

\begin{lem}\label{lem21}
 A conical region $\mathcal{C}=\mathcal{C}(M)$ is connected if each
 $\mathcal{C}_x$ is connected for any $x\in M$.
\end{lem}

The above lemma can be easily proved. Since $\mathcal{C}$ is open,
locally $\mathcal{C}$ is in the form $U\times V$, where $U$ is an
open connected neighborhood in $M$ and $V$ a  connected open cone
in $R^n$. Thus any two points of $\mathcal{C}$ can be connected by
a continuous curve by the assumption that $M$ and each
$\mathcal{C}_x$ are connected.

       A spray on $M$ (or on $\mathcal{C}(M))$ is a smooth vector field {\bf G} on a conical region $\mathcal{C}$
        of $TM \setminus \{0\}$ expressed in a local coordinate system $(x^i, y^i)$ in $TM$ as follows
    $$
    \mathbf{G} = y^i \frac{\partial}{\partial x^i} - 2G^i \frac{\partial}{\partial y^i}
    $$
    where $G^i = G^i(x, y)$ are positively homogeneous functions  of degree two.

A spray {\bf G} is called a Berwald spray  if
 the Berwald curvature   ${\bf B}=(G^{\ i}_{h\ jk})=0$, where $G^{\ i}_{h\ jk} := \dot{\partial}_h
  \dot{\partial}_j \dot{\partial}_k G^i$.
     A spray {\bf G} is said to
    be locally projectively flat if     locally  $G^i = P y^i$ everywhere.
 The  geodesics of a locally projectively flat spray  are locally straight lines.

 The Berwald connection $D$ of a spray {\bf G} is defined by
  $$
  D(\pa_i)=(G^k_{ir}dx^r)\pa_k,\ \ \ \ \ (G^k_{ir}:=\dot{\pa}_rG^k_i,\ \
  G^k_i:=\dot{\pa}_iG^k).
  $$
 For a spray tensor $T=T_idx^i$ as an example,   the horizontal and vertical
 derivatives of $T$ with respect to the Berwald
 connection are given by
  $$
 T_{i;j}=\delta_jT_i-T_rG^r_{ij},\ \ \ \ \ \ \
 T_{i.j}=\dot{\pa}_jT_i,\ \ \ \ (\delta_i:=\pa_i-G^r_i\dot{\pa}_r).
  $$
  For the Ricci identities and Bianchi identities of $D$, one
 can refer to \cite{AIM}.

The Riemann curvature tensor $R^i_{\ k}$ of {\bf G} is defined by
    \be\label{PreE1}
    R^i_{\ k} := 2\partial_k G^i - y^j (\partial_j G^i_k) +
    2G^j G^i_{jk} - G^i_j G^j_k,
    \ee
    where we put $\partial_k := \partial / \partial x^k$.
The Ricci curvature $Ric$ is defined by  $Ric := R^i_{\ i}$. A
spray {\bf G} is said to
     be $R$-flat if $R^i_{\ k} = 0$.

  A spray {\bf G}  of scalar
     curvature satisfies (\ref{yg1}).
 If $R_{.i} = 2\tau_i$ in (\ref{yg1}), then {\bf G} is said
    to be of isotropic curvature (\cite{LS}). A spray {\bf G} is said to be
     of constant curvature if {\bf G} satisfies (\ref{yg1})  with (\cite{Yang2})
    $$
    \tau_{i;k} = 0 \ (\Leftrightarrow \ R = \tau_k = 0, \text{ or } R_{;i} = 0 \ (R \neq 0).
   $$

A two-dimensional spray is always of scalar curvature, and for a
two-dimensional spray, there is an important quantity named
Berwald-Weyl curvature ${\bf W}^o$ which is projectively invariant
and defined by
 \be\label{Berw}
  {\bf W}^o=(R_{.i}+\tau_{i})_{;j}-(R_{.j}+\tau_{j})_{;i}.
 \ee

\begin{lem}(\cite{Ber})\label{Prelem2}
A two-dimensional spray is locally projectively flat if and only
if $${\bf D}={\bf W}^o=0.$$
\end{lem}

For higher dimensions in the  above lemma, ${\bf W}^o$ is replaced
by the Weyl curvature.

\begin{Def} (cf. \cite{AIM, Shen6})
A function $L=L(x,y)(\ne 0)$ is called a Finsler metric
  on a manifold $M$ if

 {\rm (i)}
   $L$ is defined on a conical region $\mathcal{C}$ of  $TM\setminus \{0\}$
  and $L$ is $C^{\infty}$;

 {\rm (ii)} $L$
 is positively homogeneous of degree two;

{\rm (iii)} the fundamental metric tensor $g_{ij}:=(L/2)_{y^iy^j}$
is non-degenerate.
\end{Def}

 If a Finsler metric $L>0$, we
put $L=F^2$, and in this case, $F$ is also called a Finsler metric
and $F$ is positively homogeneous of degree one.

 Any Finsler metric $L$ induces
a natural spray whose coefficients $G^i$ are given by
 \be\label{Gis}
 G^i:=\frac{1}{4}g^{il}\big \{L_{x^ky^l}y^k-L_{x^l}\big
 \},
 \ee
where $(g^{ij})$ is the inverse of $(g_{ij})$. $L$ is said to be
of {\it scalar flag curvature} $K=K(x,y)$ if
 \be\label{SCFF}
 R^i_{\ k}=K(L\delta^i_k-y^iy_k),\ \ (y_k:=g_{km}y^m).
 \ee
 If $K_{.i}=0$, then $L$ is said to be of
 {\it isotropic flag curvature}. $L$ is said to be of
 {\it constant flag curvature} if $K$ is constant. A Finsler metric is of scalar (resp. isotropic,
     constant) flag curvature if and only if its induced spray is of scalar (resp. isotropic,
      constant) curvature (\cite{Shen6, LS, Yang2}).

\begin{Def}\label{def1}
 Let $(M,{\bf G})$ be a spray manifold with {\bf G} being defined on a conical region
 $\mathcal{C}=\mathcal{C}(M)$ of $TM\setminus \{0\}$.
  \ben
 \item[{\rm (i)}]  {\bf G} is said to be globally
metrizable on $M$ (or on $\mathcal{C}(M)$) if there is a Finsler
metric $L$ defined on
  $\mathcal{C}(M)$ and $L$ induces {\bf G}.
  \item[{\rm (ii)}]   {\bf G} is said to be
 locally metrizable on $M$ if for each $(x,y)\in \mathcal{C}$, there is
 a conical neighborhood $U(\subset \mathcal{C})$ of $(x,y)$ such that ${\bf G}|_U$ is
 induced by a Finsler metric  on $U$.
 \een
 \end{Def}

\begin{lem}(\cite{Yang2})\label{Yplem}
Let $T$ be a positively homogeneous function of degree zero on a
spray manifold of scalar curvature with nonzero Ricci curvature.
If $T_{;i}=0$, then $T$ is  constant.
\end{lem}

\begin{lem}\label{Plem1} (\cite{BM2})
Let ${\bf G}$ be a spray  of scalar curvature $R^i_{\
k}=R\delta^i_k-\tau_ky^i$ with $R\ne 0$. Then {\bf G} is locally
metrizable if and only if  (\ref{Met1}) holds. In this case, {\bf
G} is locally induced by the following metric
 \beq
 &&\hspace{1cm}L(x,y)= c e^{2Q},\label{Met6}\\
&&\Big(\frac{\tau_{i}}{R} = Q_{.i}, \quad Q_{;i} =
   -\frac{1}{2}(\ln|c|)_{;i}\Big).\label{Met5}
 \eeq

\end{lem}

\section{Proofs of Theorem \ref{Thm01} and \ref{Thm02}}

In this section, we will give the proof of Theorem \ref{Thm01} and
Theorem \ref{Thm02}.

\

\noindent{\it Proof of Theorem \ref{Thm01} :}

 We only need to prove that if {\bf G} is
locally metrizable, then {\bf G} is globally metrizable.

It is known that (cf. \cite{Wh}) a smooth manifold (satisfying
second countable axiom)
    admits a good open covering $\mathcal{U}=\{U_i\ |\ i\in\Lambda\}$,
     which means that  every non-empty intersection of finitely many of the $U_i$'s is contractible.

\begin{lem}\label{lem0031}
 Let $\widetilde{M}$ be a manifold with trivial first de Rham cohomology
 group.  For a good open covering $\mathcal{U}=\{U_i\ |\
 i\in\Lambda\}$ of $\widetilde{M}$, suppose that on each  $U_i\cap U_j$ (not empty), we
 have a constant $\tau_{ij}$ such that
  \be\label{Sec3E3}
   \tau_{ij} + \tau_{jk} + \tau_{ki} = 0.
   \ee
  Then there exists a constant $\tau_i$ on each $U_i$ such that
   \be\label{Sec3E004}
     \tau_{ij} = \tau_i -\tau_j.
     \ee
\end{lem}

{\it Proof :}
 Due to the quantities $\tau$ in (\ref{Sec3E3}), we
  consider the constant presheaf
  with the common addition group
 in {\bf R}, which assigns to every open set $V$ of  $\widetilde{M}$ a
 locally constant function $V\mapsto {\bf R}$. Then we
  have the C$\breve{e}$ch cohomology group
 $H^1_{Ce}(\mathcal{U},{\bf R})$.

   Since $\mathcal{U}$ is a good covering, the first de Rham cohomology group $H^1_{\mathrm{dR}}(\widetilde{M})$
   is isomorphic to the first C$\breve{e}$ch cohomology group
   $H^1_{Ce}(\mathcal{U},{\bf R})$. By the assumption that
   $H^1_{\mathrm{dR}}(\widetilde{M})=0$, we have
   $H^1_{Ce}(\mathcal{U},{\bf R})=0$. This shows that any
   C$\breve{e}$ch 1-cocycle is a cobundary. It follows from (\ref{Sec3E3})
   that $\{\tau_{ij}\}$ is a 1-cocycle. Thus $\{\tau_{ij}\}$ is a
   cobundary, which implies that
    there exists constant $\tau_i$ on each $U_i$  such that
    (\ref{Sec3E004}) holds.  \qed

\

 Assume that  $\mathbf{G}$ is locally metrizable on $T^oM$. For an
 arbitrarily fixed point $x\in M$, it is clear that there is a
 neighborhood $\widetilde{U}$ of $x$ such that
 $T^o\widetilde{U}=\widetilde{U}\times R^n_o$, where
 $R^n_o:=R^n-\{0\}$.
By Definition \ref{def1}(ii),  around each point $(x,y)\in
T^o\widetilde{U}$ (where $x$ is fixed), there exists a conical
neighborhood $\mathcal{C}_{(x,y)}:=U_{(x,y)}\times V_y$
    on which some Finsler metric $L_{(x,y)}>0$ induces $\mathbf{G}$ on
    $\mathcal{C}_{(x,y)}$, where $U_{(x,y)}(\subset \widetilde{U})$
    including $x$ is connected and $(y\in)V_y$ is a connected conical region in
    $R^n_o$. It is clear that $\{V_y\big|y\in R^n_o\}$  and $\{V_y\cap S^{n-1}\big|y\in R^n_o\}$
    form two open
    coverings of $R^n_o$ and
     the unit sphere $S^{n-1}$ of
    $R^n_o$ respectively. Since $S^{n-1}$ is compact, $\{V_{y_{(i)}}\big|y_{(i)}\in
    R^n_o,i=1,\cdots,
    m\}$ for some finite integer $m$ forms a finite open covering
    of $R^n_o$.
 Put $U_x:=\cap_{i=1}^m U_{(x,y_{(i)})}$, which is an open
 neighborhood of $x$ (we may assume that $U_x$ is simply connected). Now
  $T^oU_x=U_x\times R^n_o$ is simply connected (since $n\ge 3$).
  Further, we can let $\{V_{y_{(i)}},1\le i\le m\}$ be a good open
  covering of $R^n_o$. Then
 $\{U_x\times V_{y_{(i)}},1\le i\le m\}$ is a good open covering of
 $T^oU_x$.

 Now as shown above, on each $U_x\times V_{y_{(i)}}$, we have a
 Finsler metric $L_{(x,y_{(i)})}>0$ inducing the spray {\bf G}.  Since $Ric\ne 0$ almost everywhere
  on $T^oU_x$, it follows from Lemma \ref{Yplem} and continuity that
on $U_x \times (V_{y_{(i)}}\cap V_{y_{(j)}})$,
    \be\label{Sec3E01}
    L_{(x,y_{(j)})} = \lambda_{ij} L_{(x,y_{(i)})},\ \ \ (for\ some\
    constant\ \lambda_{ij}>0).
    \ee
 It is clear from (\ref{Sec3E01}) that
    $\lambda_{ij}\lambda_{ji}=1$.
On $U_x \times (V_{y_{(i)}}\cap V_{y_{(j)}}\cap V_{y_{(k)}})$
(non-empty), (\ref{Sec3E01}) also implies
 \be\label{Sec3E2}
 \lambda_{ij}
\lambda_{jk} = \lambda_{ik}.
 \ee
Further, let
     $\tau_{ij}: = \ln \lambda_{ij}.$
     Then  (\ref{Sec3E2}) gives (\ref{Sec3E3}).
Thus, on the manifold $T^oU_x$, by Lemma \ref{lem0031}, it follows
from (\ref{Sec3E3}) that
    there exist constant $\tau_i$ on each $U_x \times V_{y_{(i)}}$  such that
     \be\label{Sec3E4}
     \tau_{ij} = \tau_i -\tau_j, \ \ {\rm or\ equivalently}\ \  \lambda_{ij} =
     e^{\tau_i}/ e^{\tau_j}.
     \ee
Thus on $U_x \times (V_{y_{(i)}}\cap V_{y_{(j)}})$, it follows
from (\ref{Sec3E01}) and (\ref{Sec3E4}) that
 \be\label{Sec3E07}
e^{\tau_j}L_{(x,y_{(j)})}=e^{\tau_i}L_{(x,y_{(i)})}.
 \ee
 Therefore, On each  $U_x \times V_{y_{(i)}}$, define a new
  function $\widetilde{L}_i
:=e^{\tau_i}L_{(x,y_{(i)})}$ which also induces {\bf G} on $U_x
\times V_{y_{(i)}}$ since $L_{(x,y_{(i)})}$ does.
   It is clear from (\ref{Sec3E07}) that  $\widetilde{L}_i$ and
$\widetilde{L}_j$ are equal on $U_x \times (V_{y_{(i)}}\cap
V_{y_{(j)}})$. Thus, we obtain a globally defined Finsler metric
$L_x>0$ on $U_x\times R^n_o$, which coincides with each
$\widetilde{L}_i$ on  $U_x \times V_{y_{(i)}}$. Clearly, $L_x$
induces {\bf G} on $T^oU_x$.

 Now $\widetilde{\mathcal{U}}=\{U_x,x\in M\}$
 forms a open covering of $M$. Take a sub-open covering
 $\mathcal{U}=\{U_i\ |\ i\in\Lambda\}$ of $M$ from $\widetilde{\mathcal{U}}$.
    Without loss of generality, assume that $\mathcal{U}$ is a good covering of $M$.
    As shown above, we have a Finsler metric $L_i>0$ on $T^oU_i$   which induces
     $\mathbf{G}$ on $T^oU_i$.

    For any non-empty intersection $U_{ij}:=U_i \cap U_j$, both $L_i$ and $L_j$ induce the same spray $\mathbf{G}$
     on $T^oU_{ij}$.
Since $U_{ij} $ is connected by assumption, $T^oU_{ij}$ is also
connected. Since $Ric\ne 0$ almost everywhere on $T^oU_{ij}$, it
follows from Lemma \ref{Yplem} and continuity that on $T^oU_{ij}$,
   $$
    L_j = \lambda_{ij} L_i, \ \ \ (for\ some\
    constant\ \lambda_{ij}>0).
   $$
Repeat a similar proof of starting from (\ref{Sec3E01}) to obtain
a Finsler metric on  $U_x \times R^n_o$. Then on the manifold $M$,
using Lemma \ref{lem0031} again, we obtain a Finsler metric $L$ on
$T^oM$, which is a  constant multiple of each local function
   $L_i$ on $T^oU_i$. Thus $L$  induces {\bf G} on
   $T^oM$ since each $L_i$ does on
   $T^oU_i$.   \qed

\begin{rem}\label{Rem31}
  The above proof is not applicable for the case $n=2$ in Theorem
  \ref{Thm01}, since $T_x^oM$ ($x\in M$) has non-trivial first de Rham cohomology
 group. It is also not applicable for the case that {\bf G} is a singular spray defined on a
connected conical region  of $T^oM$ (with $n\ge 2$), since the
neighborhood $U_x$ in the above proof might collapse to the point
$x$.
\end{rem}

\noindent{\it Proof of Theorem \ref{Thm02} :}

We only need to prove that if (\ref{Met1}) holds, then {\bf G} is
globally metrizable on $\mathcal{C}=\mathcal{C}(M)$. We follow the
main steps of the proof in \cite{BM2} for the local metrizability
of {\bf G} under the condition (\ref{Met1}), and make necessary
adjustment somewhere such that {\bf G} is globally metrizable on
$\mathcal{C}=\mathcal{C}(M)$.

The first condition in (\ref{Met1}) implies that $(\tau_i/R)dy^i$
is a closed 1-form in terms of the variable $y$. For an arbitrary
point $x_o\in M$ and a small neighborhood $U_{x_o}$ including
$x_o$, since each conical region $\mathcal{C}_x$ for $x\in
U_{x_o}$ is simply connected, it follows from the construction
proof of Poincare Lemma that there is  a smooth function
$Q=Q(x,y)$ defined on $\mathcal{C}(U_{x_o})$ such that
 $$d_yQ=(\tau_i/R)dy^i,\ \ \ or \ \ \ Q_{.i}=\tau_i/R.$$
The second condition in (\ref{Met1}) implies that $Q_{;i}$ is
independent of $x\in U_{x_o}$ and $Q_{;i}dx^i$ is a closed 1-form
on $U_{x_o}$. We may assume that $U$ is simply connected. Then by
Poincare Lemma, there is a scalar function $c=c(x)> 0$ on
$U_{x_o}$ satisfying
 $
 2Q_{;i}=-(\ln c)_{;i}.
 $
Define a function $L$ on $\mathcal{C}(U_{x_o})$ by
 $
 L:=ce^{2Q}.
 $
Then $L$ satisfies $L_{.0}=2L$ and $L_{;i}=0$, which, together
with the third condition in (\ref{Met1}), shows  that $L$ is a
Finsler metric inducing {\bf G} on $\mathcal{C}(U_{x_o})$.

As shown above, for each $x_o\in M$, we have a small neighborhood
$U_{x_o}$. In this way, we  can have a good open covering of $M$
denoted by $\mathcal{U}=\{U_i\ |\ i\in\Lambda\}$, where each $U_i$
is simply connected, and on $\mathcal{C}(U_i)$ is defined a
Finsler metric $L_i$. For any non-empty intersection $U_{ij}:=U_i
\cap U_j$, both $L_i$ and $L_j$ induce the same spray $\mathbf{G}$
     on $\mathcal{C}(U_{ij})$. Thus we get
 $$
    L_j = \lambda_{ij} L_i, \ \ \ (for\ some\
    constant\ \lambda_{ij}>0).
   $$
Then similar to the  last paragraph in the proof of Theorem
\ref{Thm01}, we obtain a Finsler metric $L$ on $\mathcal{C}(M)$,
which is a constant multiple of each local function
   $L_i$ on $\mathcal{C}(U_i)$. Therefore, $L$  induces {\bf G} on
   $\mathcal{C}(M)$.    \qed

    \section{Two-Dimensional Berwald Sprays}

 In this section, we consider the global
metrizability of a class of two-dimensional Berwald sprays on a
conical region with multiple connected components. Each spray {\bf
G} of this class on a manifold $M$ has the following local form on
some local coordinate neighborhood $(U,\{x^ i\})$ everywhere on
$M$:
 \be\label{TDBS00}
    G^1  = p_1 (y^1)^2, \ \ \ G^2  = \frac{p_2 }{\kappa}(y^2)^2,
 \ee
 where $p=p(x^1,x^2)$ is a local function in $U$ and $p_i:=p_{x^i}$,
 and $\kappa$ is a constant on $M$. We will also consider some
  special curvature properties of {\bf G} when $\kappa=-2$. The
 spray {\bf G} in (\ref{TDBS00}) is derived from the following
 lemma.

    \begin{lem}\label{ProMet2}
 Let {\bf G} be a two-dimensional Berwald spray on an open set $U$ of
 $R^2$ given by
  \be\label{TDBS1}
 G^1=\phi(x^1,x^2)(y^1)^2,\ \ \ \ \ G^2=\psi(x^1,x^2)(y^2)^2.
  \ee
 Let $Ric\ne 0$ almost everywhere on $U$.  Where $Ric\ne 0$, if {\bf G} is locally
 metrizable, then
  $$
 \phi_{x^2}=\kappa \psi_{x^1},\ \ (\kappa=constant\ne 0,-1\ and \ \psi_{x^1}\ne
 0),
 $$
 or there is a function \(p=p(x^1, x^2)\)  satisfying
 \be\label{Mety11}
  p_1 = \phi, \ \ \ p_2  = k\psi,\ \ \
   (p_i=p_{x^i}).
  \ee
Conversely, if (\ref{Mety11}) holds with $\kappa\ne 0,-1$ on $U$,
then {\bf G} is
 metrizable on $U$,
and in this case,  {\bf G} on $U$ can be induced by the following
 Finsler metric
 \be\label{Met0y11}
 L=e^{\frac{4p}{1+\kappa}}\big[(y^1)^2|y^2|^{2\kappa}\big]^{\frac{1}{1+\kappa}}.
 \ee

\end{lem}

{\it Proof :} Let {\bf G} be locally metrizable where $Ric\ne 0$.
We will use Lemma \ref{Plem1} to determine $\phi$ and $\psi$  and
 the corresponding local Finsler metric inducing {\bf G}.

 By a direct
calculation, we have $R^i_{\ k}=R\delta^i_k-\tau_ky^i$ with
 \be\label{Met00011}
 R=-2(\phi_2+\psi_1)y^1y^2,\ \ \ \ \tau_1=-2\psi_1y^2,\ \ \
 \tau_2=-2\phi_2y^1,
 \ee
 where  $\phi_i:=\phi_{x^i},\psi_i:=\psi_{x^i}$.
 Similarly, we will use $\phi_{ij}:=\phi_{x^ix^j}$ etc. in the
 following.

By (\ref{Met00011}) and the third condition in (\ref{Met1}), we
have $\psi_1\phi_2\ne0$. By (\ref{Met00011}), the first condition
in (\ref{Met1}) holds naturally.
 Further, by an integration, we easily obtain
a function $Q$ defined by
 \be\label{Met0011}
 Q:=\frac{\psi_1\ln|y^1|+\phi_2\ln|y^2|}{\phi_2+\psi_1},
 \ee
which satisfies $\tau_i/R=Q_{.i}$. The second condition in
(\ref{Met1}) gives
 $$
\phi_{12}\psi_1=\phi_2\psi_{11},\ \ \ \
\phi_2\psi_{12}=\phi_{22}\psi_1,
 $$
which gives $\phi_2=\kappa \psi_1$ for a constant $\kappa$ since
$\psi_1\ne 0$. We have $\kappa\ne -1$ since $Ric\ne 0$.

 By (\ref{Met0011}) and
$\phi_2=\kappa \psi_1$, we obtain
 $$
Q=\frac{1}{1+\kappa}\ln\big(|y^1|\cdot|y^2|^{\kappa}\big),\ \ \ \
Q_{;1}=-\frac{2}{1+\kappa}\phi,\ \ \ \
 Q_{;2}=-\frac{2\kappa}{1+\kappa}\psi,
 $$
from which and the second formula of (\ref{Met5}), we have a
function $c$ defined by
 $$
 c:=e^{\frac{4p}{1+\kappa}},
 $$
 where $p=p(x^1, x^2)$ satisfies (\ref{Mety11}). Then by (\ref{Met6}),
 we obtain the local Finsler metric (\ref{Met0y11}) which induces {\bf G}.

Conversely, it is clear that the Finsler metric $L$ given by
(\ref{Met0y11}) induces {\bf G} on $U$ under the condition
(\ref{Mety11}) with $\kappa\ne 0,-1$. \qed

\

The metrizability of the spray (\ref{TDBS1}) is also considered in
\cite{BM2}. If a two-dimensional spray {\bf G} has the local form
(\ref{TDBS1}) everywhere on a manifold $M$ with the condition
(\ref{Mety11}), then the coordinate transformation between two
neighborhoods with non-empty intersection has a special form (see
the following two lemmas).

    \begin{lem}\label{thm2.9}
        Let $\mathbf{G}$ be a two-dimensional spray on a manifold $M$ and $\kappa\ne\pm1$
        be a constant on $M$. If
         for any point $x\in M$, there is a coordinate neighborhood of $x$ in which,
         $\mathbf{G}$ has the form  (\ref{TDBS1}) with the
         condition $\phi_2=\kappa \psi_1$, then each coordinate transformation from
$(U,\{x^i\})$ to $(\bar{U},\{\bar{x}^i\})$ preserving the form
(\ref{TDBS1}) satisfies
        \be\label{coordinate}
        \bar{x}^1 = f(x^1), \ \
        \bar{x}^2 = h(x^2).
        \ee
   \end{lem}

    {\it Proof :}
    We know that the relation of the spray coefficients under two local coordinates is:
    \beq \label{gy1}
    \bar{G}^{i}=G^{r} \frac{\partial \bar{x}^{i}}{\partial x^{r}}+\frac{1}{2}
    \bar{y}^{m} \bar{y}^{r} \frac{\partial^{2} x^{k}}{\partial \bar{x}^{m}
        \partial \bar{x}^{r}} \frac{\partial \bar{x}^{i}}{\partial x^{k}}.
    \eeq

    In the following, we denote $\frac{\partial \bar{x}^{i}}{\partial x^{j}}$ by $A^i_j$, and denote
    $\frac{\partial^{2} \bar{x}^{k}}{\partial x^{m} \partial x^{r}}$ by $A^k_{mr}$.
    Then (\ref{gy1}) gives
    \beq
    &&2\bar{G}^1 =(A^1_1 A^2_2 - A^1_2 A^2_1)^{-2}\big[ R_1(\bar{y}^1)^2 + R_2\bar{y}^1\bar{y}^2
    + R_3(\bar{y}^2)^2\big],\label{G1}\\
    &&2\bar{G}^2 = (A^1_1 A^2_2 - A^1_2 A^2_1)^{-2}\big[S_1(\bar{y}^1)^2 +S_2\bar{y}^1\bar{y}^2
    + S_3(\bar{y}^2)^2\big],\label{G2}
    \eeq
    where
    \beq
    R_1 &&\hspace{-0.6cm}=
    2\phi (A^2_2)^2 A^1_1 + 2\psi (A^2_1)^2 A^1_2 - (A^2_1)^2 A^1_{22}
    + 2A^2_2 A^2_1 A^1_{12} - (A^2_2)^2 A^1_{11}, \nonumber\\
    R_2 &&\hspace{-0.6cm}=
    -2\psi A^1_1 A^2_1 A^1_2 - 2\phi A^1_2 A^2_2 A^1_1 + A^1_1 A^2_1 A^1_{22}
    - A^1_1 A^2_2 A^1_{12} -A^1_2 A^2_1 A^1_{12} + A^1_2 A^2_2 A^1_{11}, \nonumber\\
    R_3 &&\hspace{-0.6cm}=
    2(A^1_1)^2 A^1_2 \psi + 2A^1_1 (A^1_2)^2 \phi - (A^1_1)^2 A^1_{22}+ 2A^1_1 A^1_2 A^1_{12}
     - (A^1_2)^2 A^1_{11}, \nonumber\\
    S_1 &&\hspace{-0.6cm}=
    2\psi (A^2_1)^2 A^2_2 + 2\phi (A^2_2)^2 A^2_1 - (A^2_1)^2 A^2_{22}+ 2A^2_2 A^2_1 A^2_{12}
    - (A^2_2)^2 A^2_{11}, \nonumber\\
    S_2 &&\hspace{-0.6cm}=
    -2\psi A^1_1 A^2_1 A^2_2- 2\phi A^1_2 A^2_2 A^2_1 +A^1_1 A^2_1 A^2_{22}
    - A^1_1 A^2_2 A^2_{12} - A^1_2 A^2_1 A^2_{12} +A^1_2 A^2_2 A^2_{11}, \nonumber\\
    S_3 &&\hspace{-0.6cm}=
    2(A^1_1)^2 A^2_2 \psi + 2(A^1_2)^2 A^2_1 \phi - (A^1_1)^2 A^2_{22}
    + 2A^1_1 A^1_2 A^2_{12} - (A^1_2)^2 A^2_{11}. \nonumber
    \eeq

    If $\bar{G}^1=\bar{\phi}(\bar{x}^1,\bar{x}^2)(\bar{y}^1)^2$,
    we require $R_2 = 0$ and $R_3 = 0$.
    If $A^1_1 \neq 0$ and $A^1_2 \neq 0$, solving $\phi,\psi$ from $R_2 = 0$ and $R_3 = 0$
    gives
    $$
    \phi = -\frac{1}{2} \frac{\partial}{\partial x_1} \left( \frac{A^1_1}{A^1_2} \right) \cdot \frac{A^1_2}{A^1_1},
    \hspace{1cm}
    \psi = \frac{1}{2} \frac{\partial}{\partial x_2} \left( \frac{A^1_1}{A^1_2} \right) \cdot \frac{A^1_2}{A^1_1},
    $$
    which gives $\phi_{x^2}=- \psi_{x^1}$, and thus contradicts with  $\phi_2=\kappa \psi_1 \ (\kappa\ne -1)$. So
    \be\label{A112}
    A^1_1=0, \ \ \text{or} \ \ A^1_2 = 0.
    \ee
 If $\bar{G}^2=\bar{\psi}(\bar{x}^1,\bar{x}^2)(\bar{y}^2)^2$, we have $S_1 = 0$ and $S_2 = 0$.
    If $A^2_1 \neq 0$ and $A^2_2 \neq 0$, by $S_1 = 0$ and $S_2 = 0$, similarly, we   obtain
     $\phi_{x^2}=- \psi_{x^1}$,  which contradicts with  $\phi_2=\kappa \psi_1 \ (\kappa\ne -1)$ again. Therefore,
    \beq\label{A212}
    A^2_1=0, \ \ \text{or} \ \ A^2_2 = 0.
    \eeq

    Since $\det(A^i_j)\ne0$, it follows from (\ref{A112}) and
    (\ref{A212}) that one of the following two cases holds:
  \be\label{gy9}
    \bar{x}^1 = f(x^1), \ \
    \bar{x}^2 = h(x^2)\ ;\hspace{0.6cm}or \hspace{0.6cm}
    \bar{x}^1 = h(x^2), \ \
    \bar{x}^2 = f(x^1).
    \ee
    If the former case of (\ref{gy9}) holds, then by (\ref{G1}) and (\ref{G2}), in the new
    coordinate $\{\bar{x}^i\}$, we get
    \beq\label{GG1}
    \bar{G}^1=\bar{\phi}\cdot (\bar{y}^1)^2,\ \ \ \ \
    \bar{G}^2=\bar{\psi}\cdot (\bar{y}^2)^2,
    \eeq
    where
    $$
    \bar{\phi}= \frac{2\phi f'(x^1) - f''(x^1)}{2\left(f'(x^1)\right)^2}, \ \ \
    \bar{\psi}= \frac{2h'(x^2)\psi - h''(x^2)}{2\left(h'(x^2)\right)^2},
   $$
    which gives $\bar {\phi}_{\bar{x}^2}=\kappa \bar{\psi}_{\bar{x}^1}$.
     If the latter case of (\ref{gy9}) holds, we similarly have \(\kappa\bar {\phi}_{\bar{x}^2}= \bar{\psi}_{\bar{x}^1}\),
    which does not satisfy  $\phi_2=\kappa \psi_1 \ (\kappa\ne \pm1)$. Thus, we obtain  (\ref{coordinate}).
     \qed

    \begin{lem}\label{Lem4}
       In Theorem \ref{thm2.9}, if \(\kappa=1\) on $M$, then the coordinate
        transformations preserving the form (\ref{TDBS1}) are given by (\ref{gy9}).
  \end{lem}

    The above lemma  can be directly derived from the proof of  Lemma
    \ref{thm2.9}. By virtue of (\ref{Mety11}) in Lemma \ref{ProMet2}, we consider a
    spray {\bf G} with the local form (\ref{TDBS00}) on a manifold
    everywhere. We have the following theorem on the global
    metrizability for such a spray.

    \begin{thm}\label{Thm03}
       Let {\bf G} be a spray on a manifold
       $M$ with the local form (\ref{TDBS00}) on
some local coordinate neighborhood $(U,\{x^ i\})$ everywhere on
$M$ for a local function $p$ in $U$ and a consistent constant
       $\kappa\ne 0,-1$ on $M$.
        If the first de Rham cohomology group of  $M$ is trivial, then  $\mathbf{G}$
         is globally metrizable on the conical region $\mathcal{C}$ with
         $$
         \mathcal{C}(U):=\{(y^1,y^2)\in U_x,\ x\in U\ \big|\ y^1y^2\ne
         0\}.
          $$
          In this case, the Finsler metric $L$ inducing {\bf G} on $\mathcal{C}$ is a
          1-form metric in the form
 \be\label{1form}
 L=\big(\beta^2|\gamma|^{2\kappa}\big)^{\frac{1}{1+\kappa}},
 \ee
          where $\beta,\gamma$ are two independent 1-forms on $M$.
    \end{thm}

    {\it Proof:} Without loss of generality, we may assume that we have a good covering
     $\mathcal{U}=\{U_i\ |\ i\in\Lambda\}$ of $M$, such that for
     each member  $U$ of $\mathcal{U}$ with the coordinate
     $\{x^i\}$, {\bf G} has the form (\ref{TDBS00}).

     Let $U$ and $\bar{U}$ be any two members of $\mathcal{U}$ with  non-empty intersection,
     and their  coordinates are
     $\{x^i\}$ and  $\{\bar{x}^i\}$ respectively. On  $U$ and $\bar{U}$, by the assumption of (\ref{TDBS00}), the coefficients
      of {\bf G} are respectively given by
   \beq
    G^1  \hspace{-0.6cm}&&= p_1 (y^1)^2, \ \ \ G^2  = \frac{p_2 }{\kappa}(y^2)^2,\label{GGU1}\\
    \bar{G}^1 \hspace{-0.6cm}&& = \bar{p}_1 (\bar{y}^1)^2, \ \ \ \
    \bar{G}^2  = \frac{\bar{p}_2 }{\kappa}(\bar{y}^2)^2.\label{GGU2}
    \eeq
 Then by Lemma
\ref{ProMet2}, it follows from (\ref{Met0y11}) that, on
$\mathcal{C}(U)$, {\bf G} is induced by the Finsler metric
   \be\label{FU1}
 L_U=e^{\frac{4p}{1+\kappa}}\big[(y^1)^2|y^2|^{2\kappa}\big]^{\frac{1}{1+\kappa}},
 \ee
 and on $\mathcal{C}(\bar{U})$, {\bf G} is induced by the Finsler metric
 \be\label{FU01}
  \bar{L}_{\bar{U}}=e^{\frac{4\bar{p}}{1+\kappa}}\big[(\bar{y}^1)^2|\bar{y}^2|^{2\kappa}\big]^{\frac{1}{1+\kappa}}.
 \ee

   If $\kappa\ne 1$, we have (\ref{coordinate}) by Lemma \ref{thm2.9},
   and then on $U\cap\bar{U}$,  it follows from (\ref{GG1}) and  (\ref{Mety11}) that
 \be\label{TDBS36}
\bar{p}_1=\frac{p_1}{f'}-\frac{f''}{2(f')^2},\ \ \ \
\bar{p}_2=\frac{p_2}{h'}-\frac{\kappa h''}{2(h')^2},
 \ee
    where $p_i,\bar{p}_i$ are determined by (\ref{GGU1}) and
    (\ref{GGU2}). The integration of (\ref{TDBS36}) yields
 \be\label{TDBS37}
 \bar{p} = p - \frac{1}{2} \big(\ln|f'| + \kappa \ln|h'|\big)+c,
 \ee
 where $c$ is a constant on the connected neighborhood
 $U\cap\bar{U}$. Then since (\ref{TDBS37}) and (\ref{coordinate}), it
 follows from (\ref{FU1}) and (\ref{FU01}) that
  $$
 \bar{L}_{\bar{U}}=e^{\frac{4c}{1+\kappa}}L_U,
  $$
which shows that $\bar{L}_{\bar{U}}$ is a positive constant
multiple of $L_U$ on $U\cap\bar{U}$.

    If $\kappa=1$, we have (\ref{gy9}) by Lemma \ref{Lem4}. The former case of (\ref{gy9}) has been
    considered
     in the above proof to the case  $\kappa\ne 1$. We consider the latter case  of (\ref{gy9}),
     that is,  $\bar{x}^1 = h(x^2)$, $\bar{x}^2 = f(x^1)$. Similar analysis shows that
     on the connected neighborhood $\bar{U}\cap\bar{U}$, we have
    \beqn
    &&\bar{p}_1:=\frac{p_2}{h'}-\frac{h''}{2(h')^2},
   \ \ \ \ \bar{p}_2:=\frac{p_1}{f'}-\frac{f''}{2(f')^2},\\
    &&\hspace{1.cm}\bar{p} = p - \frac{1}{2} \big(\ln|h'| + \ln|f'|\big)+c,
    \eeqn
    where $c$ is a constant on the connected neighborhood
 $U\cap\bar{U}$. Therefore, it
 follows from (\ref{FU1}) and (\ref{FU01}) with $\kappa=1$ that $\bar{L}_{\bar{U}}$ is a
 positive constant multiple of $L_U$
 on $U\cap\bar{U}$.

 Thus for  any two members $U_i$ and $U_j$ of the good covering
 $\mathcal{U}$ with $U_i\cap U_j$ being non-empty,
 the local Finsler metrics $L_i$ (on $\mathcal{C}(U_i)$) is
 a positive constant multiple $\lambda_{ij} $ of $L_j$ (on
 $\mathcal{C}(U_j)$).
 Note that although the conical region $\mathcal{C}$ typically has
    four connected components,
    the above positive constant $\lambda_{ij}$ on each $ U_i\cap U_j$ are well-defined across all components.
Therefore, following the  last paragraph in the proof of Theorem
\ref{Thm01},
    we can find a Finsler metric $L$ on $\mathcal{C}$
    inducing $\mathbf{G}$. Thus {\bf G} is globally metrizable on $\mathcal{C}$.

 Finally, it is clear from the above discussion that the Finsler metric $L$
 inducing {\bf G} on $\mathcal{C}$ is a
  1-form metric given by (\ref{1form}), where, by (\ref{FU1}), the two 1-forms
  $\beta,\gamma$ on $M$ have the local expression
   $$
 \beta:=ce^{\frac{2p}{1+\kappa}}y^1,\ \ \
 \gamma:=ce^{\frac{2p}{1+\kappa}}y^2,\ \ \ (c\ is \ a\ local\
 constant).
   $$
For this fact,  we may note that for two open neighborhoods $U,V$
of $M$ with $U\cap V$ not empty, letting
 $$L_U= \big(\hat{\beta}^2|\hat{\gamma}|^{2\kappa}\big)^{\frac{1}{1+\kappa}},\
 \ \ L_{V}= \big(\widetilde{\beta}^2|\widetilde{\gamma}|^{2\kappa}\big)^{\frac{1}{1+\kappa}},
 $$
  we then  obtain $L_V=\kappa L_U$ on $U\cap V$
for some local constant $\kappa$, and thus
$\hat{\beta},\hat{\gamma}$ and
$\widetilde{\beta},\widetilde{\gamma}$ are related by
$\widetilde{\beta}=\kappa\hat{\beta},\widetilde{\gamma}=\kappa\hat{\gamma}$.
Thus using Lemma \ref{lem0031} again,  we obtain the Finsler
metric defined by (\ref{1form}) with two independent 1-forms
$\beta,\gamma$ on $M$. \qed

\

Now we further consider  some special curvature properties of the
spray {\bf G} in (\ref{TDBS00}) when $\kappa=-2$. We first prove
the following proposition.

 \begin{prop}\label{mainthm}
  Let $(M,{\bf G})$ be  a two-dimensional spray manifold, where {\bf G} is a locally projectively
   flat Berwald spray  of almost everywhere non-isotropic
    curvature on $M$. If {\bf G} is locally metrizable on a conical region $\mathcal{C}$,
    then around any point of $M$, there is a local
    coordinate neighborhood $(U,\{x^i\})$ such that
    ${\bf G}|_U$ can be expressed as
   \be\label{GGi}
  G^1 = -\frac{1}{2} \frac{\psi_1}{\psi} (y^1)^2, \quad G^2 = \frac{1}{4}
  \frac{\psi_2}{\psi} (y^2)^2,\ \ \ \big(\psi_i:=\psi_{x^i}\big),
  \ee
   where $\psi=\psi(x^1, x^2)$ satisfies
   \be\label{GGi1}
       \psi\big(\frac{\psi_2}{\psi}\big)_1 = c
        \ee
   for a scalar function $c=c(x^2)$, and $c\ne 0$ almost everywhere in  $U$.     In this case,
   ${\bf G}|_U$ is induced by the following metric
        \be\label{GGi2}
        F = \psi(x^1, x^2) \cdot \frac{(y^2)^2}{y^1}.
        \ee
    \end{prop}

 {\it Proof :} By assumption, {\bf G} is locally metrizable. Then
 for an arbitrary point $P\in M$, there is a neighborhood $V$
 including $P$ such that ${\bf G}|_V$ is induced by a Finsler metric $F$.  Since {\bf G} is a two-dimensional locally
  projectively flat Berwald spray of almost everywhere non-isotropic curvature, $F$ is a two-dimensional locally
  projectively flat Berwald metric of almost everywhere non-isotropic flag curvature on $V$. Then by Berwald's result in
  \cite{Ber}, we use continuity to conclude that  $F$ is a 1-form metric
  $F=\beta^2/\gamma$ with vanishing Berwald-Weyl curvature
  $W^o=0$ (cf. Lemma \ref{Prelem2}), where $\beta,\gamma$ are two independent 1-forms. In a
  local coordinate system $(U,\{x^i\})$ ($U\subset V$), $\beta$ and $\gamma$ can be written as
   $$\beta=py^2,\ \ \ \gamma=qy^1,$$
for two scalar functions $p=p(x^1,x^2)$ and $q=q(x^1,x^2)$. Then
$F=\beta^2/\gamma$ has the local expression
 $$F=\frac{p^2(y^2)^2}{qy^1}=\psi\cdot \frac{(y^2)^2}{y^1},\ \ \ \ \ (\psi:=p^2/q),$$
which gives (\ref{GGi2}). Thus by (\ref{Gis}), a direct
calculation shows that the spray coefficients $G^1,G^2$ of $F$ are
given by (\ref{GGi}). By (\ref{GGi}) and (\ref{PreE1}), the
Riemann curvature $R^i_{\ k}$ of {\bf G} is given by
 $$
R^i_{\ k}=R\delta^i_k-\tau_iy^k,
 $$
where $R,\tau_1,\tau_2$ are defined by
 \beqn
  R:=\frac{\psi\psi_{12}-\psi_1\psi_2}{2\psi^2}y^1y^2,\ \ \
 \tau_1:=-\frac{\psi\psi_{12}-\psi_1\psi_2}{2\psi^2}y^2,
 \ \ \ \tau_2:=\frac{\psi\psi_{12}-\psi_1\psi_2}{\psi^2}y^1,
 \eeqn
 Then plugging the $R,\tau_1,\tau_2$ into (\ref{Berw}), we obtain
 $$
 {\bf
 W}^o=-\frac{3}{2\psi^3}\big[\psi(\psi\psi_{112}-\psi_2\psi_{11})-\psi_1(\psi\psi_{12}-\psi_1\psi_2)\big].
 $$
Since {\bf G} is of almost everywhere non-isotropic curvature,
$\psi\psi_{12}-\psi_1\psi_2 \ne 0$ almost everywhere on $U$. Then
where $\psi\psi_{12}-\psi_1\psi_2 \ne 0$, we derive that ${\bf
W}^o=0$ if and only if
 $$
\frac{(\psi\psi_{12}-\psi_1\psi_2)_1}{\psi\psi_{12}-\psi_1\psi_2}=\frac{\psi_1}{\psi},\
\ \ or\ \ equi.\ \ \psi\big(\frac{\psi_2}{\psi}\big)_1 = c,
 $$
where $c=c(x^2)$  and $c\ne 0$ almost everywhere. This gives
(\ref{GGi1}).  \qed

\

In Proposition \ref{mainthm}, if the conical region $\mathcal{C}$
is maximal, then by (\ref{GGi2}), we have
 $$
  \mathcal{C}(U):=\{(y^1,y^2)\in U_x,\ x\in U\ \big|\ y^1y^2\ne
  0\}.
  $$

    \begin{rem}\label{rm12}
The PDE (\ref{GGi1}) is solvable.
   Putting
    $$
    \psi(x^{1},x^{2})= e^{p(x^{1}, x^{2})},
    $$
    we can write (\ref{GGi1}) as
     $$
      p_{12}=c(x^{2})e^{-p},
      $$
which is the Liouville equation with the  general solution
    \beq \label{liup}
    p(x^{1},x^{2})=\ln \left( -\frac{c(x^{2}) (f(x^{1}) + g(x^{2}))^{2}}{2 f'(x^{1}) g'(x^{2})}
    \right),
    \eeq
    where $f(x^1)$ and $f(x^2)$ are arbitrary smooth functions
    satisfying certain condition such that the expression in $\ln$
    is positive.
    \end{rem}

By Theorem \ref{Thm03} and Proposition \ref{mainthm}, we have the
following corollary.

 \begin{cor}\label{cor0l}
  Let $(M,{\bf G})$ be  a two-dimensional spray manifold, where {\bf G} is a locally projectively
   flat Berwald spray  of almost everywhere non-isotropic
    curvature, and $M$ has trivial first de Rham cohomology group. If {\bf G} is
     locally metrizable on a conical region $\mathcal{C}=\mathcal{C}(M)$,
        then {\bf G} is globally metrizable on  $\mathcal{C}$. In this case,  {\bf G} on $\mathcal{C}$ is induced
        by a 1-form metric  $F=\beta^2/\gamma$  on $M$, where $\beta,\gamma$
      are two  independent 1-forms  on $M$.
    \end{cor}

{\it Proof :} Firstly, by Proposition \ref{mainthm},
 {\bf G} has the local form (\ref{GGi}) everywhere on $M$, which is a special case of
 (\ref{TDBS00}), where we put in (\ref{TDBS00}),
  $$
 p:=-\frac{1}{2}\ln|\psi|,\ \ \ \ \kappa=-2.
  $$
Therefore, {\bf G} is globally metrizable on  $\mathcal{C}$ by
Theorem \ref{Thm03}. Let $F$ be the Finsler metric inducing {\bf
G} on $M$. By (\ref{1form}), we have $F=\beta^2/\gamma$.  \qed

\

The 1-form metric $L$ in (\ref{1form}) has constant main scalar
$I$ with $\epsilon I^2=-(\kappa-1)^2/\kappa$, where $\epsilon=\pm
1$ (\cite{Ber, AIM}). Especially, we have $I^2=9/2$ for the
metric $F$ in Corollary \ref{cor0l}.

    \section{Examples of Two-Dimensional Sprays}\label{Example}

    In Theorem \ref{Thm01}, if the
    first de Rham cohomology group of $M$ is not trivial, then the spray $\mathbf{G}$
    may not be globally metrizable, although it is locally metrizable.
    We give a family of  examples on a cylinder as follows.

 \begin{ex}\label{ex001}
 Let \( M \) be the cylinder \( M=S^1 \times \mathbf{R} \), which has non-trivial first de Rham cohomology group.
 At an arbitrary point $(x^1,x^2)\in M$, let
   \beqn
   &&\phi(x^1, x^2): = c + f(\cos x^1)sin x^1\cdot h(x^2),\\
  &&\psi(x^1, x^2): = -\frac{1}{\kappa} \widetilde{f}(\cos x^1) h'(x^2),\hspace{0.9cm}
  \big(\widetilde{f}(t):=\int f(t)dt\big),
    \eeqn
    where $f,h$ are smooth functions, and $c,\kappa$ are constant with $c\ne0$ and $\kappa\ne 0,-1$.
    Define a spray {\bf G} on $M$  by
  \beqn
    &&G^1=\phi(x^1,x^2)(y^1)^2=p_1 (y^1)^2,\ \ \ \ \ G^2=\psi(x^1,x^2)(y^2)^2=\frac{p_2}{\kappa}(y^2)^2,\ \ \
    (p_i:=p_{x^i}),\\
    &&\hspace{3cm} \big(\ p(x^1,x^2):=cx^1-\widetilde{f}(\cos x^1) h(x^2)\ \big).
  \eeqn
 Note that $p$ is a local function, but {\bf G}  is globally defined on
 \(M \).

 It can be verified that {\bf G}
   is locally metrizable but not globally metrizable on the conical region $\mathcal{C}=\mathcal{C}(M)$,
   where, on each tangent space $T_xM$ with $x=(x^1,x^2)$,
   $$
 \mathcal{C}_x:=\{(y^1,y^2)\ \big|\ (y^1,y^2)\in T_xM,\ y^1y^2\ne0\}.
   $$
 We verify this fact in the following. By Lemma \ref{ProMet2}, {\bf G} is locally
 metrizable on $\mathcal{C}$. To show that {\bf G} is not globally
 metrizable on $\mathcal{C}$, consider a  covering of the cylinder with two
 coordinate neighborhoods $(U,\{x^i\})$ and
 $(\bar{U},\{\bar{x}^i\})$, where
        \[
        U = (0, 2\pi) \times \mathbf{R}, \quad \bar{U} = (-\pi, \pi) \times \mathbf{R}.
        \]
        Their intersection \( U \cap \bar{U} \) has two connected components:
        \[
        V_- = (0, \pi) \times \mathbf{R}, \quad V_+ = (\pi, 2\pi) \times \mathbf{R}.
        \]
 It is clear that
  $$
 x^1=\bar{x}^1,\ \  x^2=\bar{x}^2\ \  on \ V_-;\hspace{1cm} x^1=\bar{x}^1+2\pi,\  \ x^2=\bar{x}^2\ \  on \
 V_+.
  $$
 On $\mathcal{C}(U)$, it follows from (\ref{Met0y11}) that
         $\mathbf{G}$ is induced by the Finsler metric
        \be\label{Met000y11}
  L_U=e^{\frac{4}{1+\kappa}p(x^1,x^2)}\big[(y^1)^2|y^2|^{2\kappa}\big]^{\frac{1}{1+\kappa}},
   \ee
  Assume that there is a Finsler metric $L$ inducing {\bf G} on $M$ (or
  $\mathcal{C}(M))$. Then by Lemma \ref{Yplem} and continuity, we
  have $L=\tau L_U$ on $\mathcal{C}(U)$ for a constant $\tau\ne 0$. Thus we may let
   $L=L_U$ on $\mathcal{C}(U)$. Further, put $L_{\bar{U}}:=L$ on
   $\bar{U}$. Then since $L_{\bar{U}}=L_U$ on $V_-\cup V_+$, by (\ref{Met000y11}) we get
   $$
 L_{\bar{U}}=
 \begin{cases}
  e^{\frac{4}{1+\kappa}p(\bar{x}^1,\bar{x}^2)}\big[(\bar{y}^1)^2|\bar{y}^2|^{2\kappa}\big]^{\frac{1}{1+\kappa}},
  \hspace{1.4cm} when \ 0<\bar{x}^1<\pi, \\
e^{2\pi c}\cdot
e^{\frac{4}{1+\kappa}p(\bar{x}^1,\bar{x}^2)}\big[(\bar{y}^1)^2|\bar{y}^2|^{2\kappa}\big]^{\frac{1}{1+\kappa}},
\ \ \ \  when \ -\pi<\bar{x}^1<0,
\end{cases}
$$
It is clear that $L_{\bar{U}}$ is not continuous where
$\bar{x}^1=0$, which gives a contradiction.
    \end{ex}

By (\ref{GGi2}) and Remark \ref{rm12}, we can construct many
1-form metrics $F=\beta^2/\gamma$ globally defined on a manifold
$M$ whose first de Rham cohomology  group is non-trivial, and
meanwhile, on such a manifold $M$, this also provides some sprays
  which are  locally projectively flat, Berwaldian, of
 almost everywhere non-isotropic curvature on $M$ and globally metrizable (cf. Corollary \ref{cor0l}).

    \begin{ex}\label{exl}
        Let \( M \) be the cylinder \( S^1 \times \mathbf{R} \) with \( \theta \)
        the coordinate on the circle \( S^1 \), and \( z \) the coordinate along the real line \( \mathbf{R} \).
 By (\ref{GGi2}) and (\ref{liup}), we choose suitable functions
 $f(z),g(\theta)$ and $c(\theta)$ to define the Finsler metric
 $F$ in (\ref{GGi2}). For example, let
    $$
 f(z) = e^{z} + 1, \quad g(\theta) =\kappa \sin\theta, \quad c(\theta) =
 \cos\theta,\ \ (g'(\theta)=\kappa c(\theta)),
  $$
  where $\kappa$ is a constant with $\kappa\ne 0$ and $|\kappa|<1$.
  Then we obtain
  $$
 F= -\frac{(\kappa\sin\theta + e^{z} + 1)^2}{2\kappa e^z}
 \frac{(y^2)^2}{y^1}.
  $$
 $F$ is  globally defined  on \( M
 \), and by Proposition \ref{mainthm} and Remark \ref{rm12},   $F$ is a locally projectively
   flat Berwald metric  of almost everywhere non-isotropic
  flag  curvature on $M$.
 \end{ex}

As a complement to Corollary \ref{cor0l}, we give a family of
locally projectively flat Berwald sprays on $S^2$ which are
   of non-isotropic curvature almost everywhere, and they are not
   globally metrizable on $S^2$ and even    not locally metrizable
  anywhere on $S^2$.

    \begin{ex}
 Let $\alpha$ be the standard  Riemannian metric of constant sectional
 curvature $1$ on $S^2$, and $P$ be a 1-form on $S^2$ which is non-closed (almost everywhere).
 Define a spray {\bf G} on $S^2$ by
  $$
  G^i = G_\alpha^i + P y^i,
  $$
  where ${\bf G}_{\alpha}$ denotes the spray induced by $\alpha$.
  It is clear that {\bf G} is of scalar curvature satisfying
  (\ref{yg1}). Since $P$ and $R_{.i}-2\tau_i$
($i=1,2)$ are 1-forms on $S^2$, they
   must vanish at some point of $S^2$.
  Since $P$ is a non-Hamel function (almost everywhere) on $S^2$, it follows from
Proposition 4.4 in \cite{Yang3}  that
   {\bf G} is of non-isotropic
  curvature on $S^2$ almost everywhere.  We see that {\bf G} cannot be
  globally metrizable on $S^2$ by Corollary \ref{cor0l}  (no independent 1-forms
  $\beta,\gamma$ which are nowhere vanishing on $S^2$).

 We give a concrete construction of the 1-form $P$ as above.
 The basic idea is to start from a nowhere closed 1-form $\omega=\omega_1du^2+\omega_2du^2+\omega_3du^2$ on ${\bf R}^3$,
  let $P$ be the restriction of $\omega$ on $S^2$.
   For example, let
    $$
     \omega=-u^2 du^1 + u^1 du^2\ \ on\ \ \mathbf{R}^3.
     $$
   $P$ is actually not closed anywhere on $S^2$. Further,
it can be shown that on the semi-sphere
$u^3=\pm\sqrt{1-(u^1)^2-(u^2)^2}\ne0$, the 1-forms
$R_{.i}-2\tau_i$ ($i=1,2)$ do not vanish anywhere. Now we give the
details of computation on the semi-sphere
$u^2=\sqrt{1-(u^1)^2-(u^2)^3}>0$.  Take the center projection
coordinate $(x^1,x^2)$ determined by
 $$
 u^1=\frac{x^1}{T},\ \ \
 u^3=\frac{x^2}{T},\ \ \
 u^2=\frac{1}{T},\ \ \ \big(T:=\sqrt{1+(x^1)^2+(x^2)^2}\ \big).
 $$
 Then the local expression of $P$ is given by
  $$
 P=P(x,y)=-\frac{y^1}{1+(x^1)^2+(x^2)^2}.
  $$
Thus, under the coordinate $(x^1,x^2)$, we have
 $$
 G^i=Qy^i,\ \ \ Q:=\frac{-(1+x^1)y^1-x^2y^2}{1+(x^1)^2+(x^2)^2}.
 $$
 Put $R^i_{\ k}=R\delta^i_k-\tau_ky^i$ for the Riemann curvature of
 {\bf G}. Then a direct computation gives
  $$
 R_{.1}-2\tau_1=-\frac{6x^2y^2}{\big[1+(x^1)^2+(x^2)^2\big]^2},\
 \  \ R_{.2}-2\tau_2=\frac{6x^2y^1}{\big[1+(x^1)^2+(x^2)^2\big]^2},
  $$
 So $R_{.i}-2\tau_i=0$ when $x^2=0$, from which, we can conclude that
 $R_{.i}-2\tau_i=0$ only on the great circle $u^3=0,(u^1)^2+(u^2)^2=1$.
  Thus {\bf G} is non-isotropic almost
  everywhere on $S^2$. Further, we can show that $(\tau_i/R)_{;j}\ne 0$ almost
  everywhere on $S^2$. So {\bf G} is actually  not locally metrizable
  anywhere on $S^2$ by Lemma \ref{Plem1}.
    \end{ex}

\begin{ex}
 In Proposition \ref{mainthm}, if we write {\bf G} in the form
 $G^i=Py^i$ under the classical locally projectively flat coordinate system, we can
 use Lemma \ref{Plem1} to give some solutions of $P$ such that
 {\bf G} is metrizable. For example, let
   $$
  G^i = Py^i,\ \ \ \
   P = \frac{\big(\sqrt{(x^{1})^{2}-2x^{2}}\, - x^{1}\big)y^{1} + 2y^{2}}
   {\sqrt{(x^{1})^{2}-2x^{2}}\cdot\bigl(\sqrt{(x^{1})^{2}-2x^{2}}-x^{1}\bigr)},
   $$
 and {\bf G} is induced by the metric $F=\beta^2/\gamma$ with
  \beqn
   &&\beta:=\bigl(\sqrt{(x^{1})^{2}-2x^{2}}-x^{1}\bigr)^{-4}\big[(-\sqrt{(x^{1})^{2}-2x^{2}}+x^1)y^{1}-y^{2}\big],\\
   &&\gamma:=\bigl(\sqrt{(x^{1})^{2}-2x^{2}}-x^{1}\bigr)^{-4}
    (-x^{1}y^{2}+x^{2}y^{1}).
   \eeqn
 The metric $F$ must be  of course a 1-form metric, which coincides with
 the metric type in Proposition \ref{mainthm}.
    \end{ex}

{\bf Conflict of Interest}: The authors declare no conflict of
interest. The manuscript has no associated data.

    \vspace{0.6cm}

    \noindent Xingzhi Duanmu   and Guojun Yang\\
    Department of Mathematics \\
    Sichuan University \\
    Chengdu 610064, P. R. China \\
    duanmuxingzhi@163.com\\
    yangguojun@scu.edu.cn,
\end{document}